\documentclass{amsart}
\usepackage{amssymb}
\usepackage{bm}
\usepackage{graphicx}
\usepackage[centertags]{amsmath}
\usepackage{amsfonts}
\usepackage{amsthm}
\newtheorem{thm}{Theorem}
\newtheorem{cor}[thm]{Corollary}
\newtheorem{lem}[thm]{Lemma}

\newtheorem{claim}[thm]{Claim}
\newtheorem{subclaim}[thm]{Subclaim}
\newtheorem{fact}[thm]{Fact}
\newtheorem{defn}[thm]{Definition}
\theoremstyle{definition}

\newcommand{\rr}{\mathbb{R}}
\newcommand{\nn}{\mathbb{N}}
\newcommand{\qq}{\mathbb{Q}}
\newcommand{\ee}{\varepsilon}

\newcommand{\tr}{\mathrm{Tr}}
\newcommand{\wf}{\mathrm{WF}}
\newcommand{\bt}{\nn^{<\nn}}
\newcommand{\ibt}{[\nn]^{<\nn}}
\newcommand{\bs}{\nn^\nn}
\newcommand{\sg}{\sigma}
\newcommand{\seg}{\mathfrak{s}}
\newcommand{\con}{\smallfrown}

\newcommand{\aaa}{\mathcal{A}}
\newcommand{\bbb}{\mathcal{B}}

\newcommand{\fff}{\mathcal{F}}
\newcommand{\gggg}{\mathcal{G}}
\newcommand{\hhh}{\mathcal{H}}
\newcommand{\llll}{\mathcal{L}}
\newcommand{\ppp}{\mathcal{P}}
\newcommand{\sss}{\mathcal{S}}
\newcommand{\rs}{\mathrm{r}_{\sss}}
\newcommand{\irm}{\mathrm{I}}

\newcommand{\supp}{\mathrm{supp}}

\newcommand{\sspan}{\mathrm{span}}

\begin{document}

\title{On strictly singular operators between separable Banach spaces}
\author{Kevin Beanland and Pandelis Dodos}

\address{Department of Mathematics and Applied Mathematics, Virginia Commonwealth
University, Richmond, VA 23284.}
\email{kbeanland@vcu.edu}

\address{Department of Mathematics, University of Athens, Panepistimiopolis 157 84, Athens, Greece.}
\email{pdodos@math.ntua.gr}

\thanks{The second named author was supported by NSF grant DMS-0903558.}
\maketitle


\begin{abstract}
Let $X$ and $Y$ be separable Banach spaces and denote by $\sss\sss(X,Y)$
the subset of $\llll(X,Y)$ consisting of all strictly singular operators.
We study various ordinal ranks on the set $\sss\sss(X,Y)$. Our main results
are summarized as follows. Firstly, we define a new rank $\rs$ on $\sss\sss(X,Y)$.
We show that $\rs$ is a co-analytic rank and that dominates the
rank $\varrho$ introduced by Androulakis, Dodos, Sirotkin and Troitsky
[Israel J. Math., 169 (2009), 221-250]. Secondly, for every $1\leq p<+\infty$
we construct a Banach space $Y_p$ with an unconditional basis such that
$\sss\sss(\ell_p, Y_p)$ is a co-analytic non-Borel subset of $\llll(\ell_p,Y_p)$
yet every strictly singular operator $T:\ell_p\to Y_p$ satisfies $\varrho(T)\leq 2$.
This answers a question of Argyros.
\end{abstract}


\section{Introduction}

An operator $T:X\to Y$ between two infinite-dimensional Banach spaces $X$ and $Y$
is said to be \textit{strictly singular} if the restriction of $T$ on every
infinite-dimensional subspace $Z$ of $X$ is not an isomorphic embedding (throughout
the paper by the term \textit{operator} we mean bounded linear operator;
all Banach spaces are over the real field). This is a wide class of operators
between Banach spaces that includes the compact ones. A number of discoveries,
especially after the work of W. T. Gowers and B. Maurey
\cite{GM}, have revealed the critical role of strictly singular operators on the
structure theory of general Banach spaces.

Notice that an operator $T:X\to Y$ is strictly singular if and only if
for every normalized basic sequence $(x_n)$ in $X$ and every $\ee>0$
there exist a non-empty finite subset $F$ of $\nn$ and a norm-one
vector $x\in\sspan\{x_n:n\in F\}$ such that $\|T(x)\|\leq\ee$.
This equivalence gives us no hint of where the set $F$ is located and,
in particular, of how ``difficult" it is to find it. Recently, the notion
of a strictly singular operator was refined in order to measure this difficulty.
The refinement was achieved with the use of the Schreier families
$\sss_\xi$ $(1\leq \xi<\omega_1)$ introduced in \cite{AA}.
\begin{defn}[\cite{ADST}] \label{1d1}
Let $X, Y$ be infinite-dimensional Banach spaces,
$T\in\llll(X,Y)$ and $1\leq \xi<\omega_1$.
The operator $T$ is said to be \emph{$\sss_\xi$-strictly singular}
if for every normalized basic sequence $(x_n)$ in $X$ and every
$\ee>0$ there exist a non-empty set $F\in\sss_\xi$ and a norm-one vector
$x\in\sspan\{x_n:n\in F\}$ such that $\|T(x)\|\leq\ee$.

For every $T\in\llll(X,Y)$ we set
\begin{equation} \label{1e1}
\varrho(T)= \inf\{\xi: T \text{ is } \sss_\xi\text{-strictly singular}\}
\end{equation}
if $T$ is $\sss_\xi$-strictly singular for some $1\leq \xi<\omega_1$;
otherwise we set $\varrho(T)=\omega_1$.
\end{defn}
A basic fact, proved in \cite{ADST}, is that if $X$ and $Y$ are separable,
then an operator $T:X\to Y$ is strictly singular if and only if
$\varrho(T)<\omega_1$ (this equivalence fails if $X$ and $Y$ are
non-separable; see \cite{ADST}). In particular, the map $T\mapsto\varrho(T)$
is an ordinal rank\footnote[1]{Ordinal ranks are standard tools in Banach Space
Theory; see, for instance, \cite{AGR,Bou,O,Sz}.} on the set $\sss\sss(X,Y)$
of all strictly singular operators from $X$ to $Y$. It was
further studied in \cite{AB,Be,CFPTT,P}.

In the present paper we continue the study of the rank $\varrho$
by focusing on its global properties. These kind of questions are
naturally studied within the framework of Descriptive Set
Theory (we briefly recall in \S 2.1, \S 2.2 and \S 2.3
all concepts from Descriptive Set Theory related to our work).
To put things in a proper perspective, let us first notice that
if $X$ and $Y$ are separable Banach spaces, then the set $\llll(X,Y)$
carries a natural structure of a standard Borel space (see \S 2.2)
and it is easy to see that $\sss\sss(X,Y)$ is a co-analytic
subset of $\llll(X,Y)$. Once the proper framework has been set up,
a basic problem is to decide whether the rank $\varrho$ is actually
a co-analytic rank on the set $\sss\sss(X,Y)$. Co-analytic ranks
are fundamental tools in Descriptive Set Theory and have proven
to be extremely useful in studying the geometry of Banach spaces
(see, for instance, \cite{AD,AGR,Bos,Dnew,D,DF}).

As we shall see, the rank $\varrho$ is not, in general, a co-analytic
rank. Our first main result shows, however, that it is always
sufficiently well-behaved.
\begin{thm} \label{1t2}
Let $X$ and $Y$ be separable Banach spaces. Then there exists
a co-analytic rank $\rs:\sss\sss(X,Y)\to\omega_1$ such that
\begin{equation} \label{1e2}
\varrho(T)\leq \rs(T)
\end{equation}
for every strictly singular operator $T:X\to Y$.

In particular, the rank $\varrho$ satisfies boundedness; that is,
if $\aaa$ is an analytic subset of $\sss\sss(X,Y)$, then
$\sup\{\varrho(T):T\in\aaa\}<\omega_1$.
\end{thm}
As a consequence we get the following.
\begin{cor}[\cite{Be}] \label{1c3}
If $X$ and $Y$ are separable Banach spaces and $\sss\sss(X,Y)$
is a Borel subset of $\llll(X,Y)$, then
$\sup\{\varrho(T):T\in\sss\sss(X,Y)\}<\omega_1$.
\end{cor}
A natural problem, originally asked by S. A. Argyros, is whether
the converse of Corollary \ref{1c3} is true. In particular,
it was conjectured in \cite[\S 4.5]{Be} that if $X$ and $Y$
are separable Banach spaces and
$\sup\{\varrho(T):T\in\sss\sss(X,Y)\}<\omega_1$, then
$\sss\sss(X,Y)$ is a Borel subset of $\llll(X,Y)$. Our
second main result answers this question in the negative.
\begin{thm} \label{1t4}
Let $1\leq p<+\infty$. Then there exists a Banach space
$Y_p$ with an unconditional basis such that the following
are satisfied.
\begin{enumerate}
\item[(i)] The set $\sss\sss(\ell_p,Y_p)$ is a complete
co-analytic (in particular non-Borel) subset of $\llll(\ell_p,Y_p)$.
\item[(ii)] If $T:\ell_p\to Y_p$ is strictly singular, then $\varrho(T)\leq 2$.
\end{enumerate}
In particular, the rank $\varrho$ is not a co-analytic rank
on $\sss\sss(\ell_p, Y_p)$.
\end{thm}
The paper is organized as follows. In \S 2 we gather some
background material. In \S 3 we give a criterion for checking that
$\varrho(T)\leq\xi$ when the spaces $X$ and $Y$ have Schauder bases.
In \S 4 we give the proof of Theorem \ref{1t2}. In \S 5 we introduce
a class of spaces $Z_{p,q}$ $(1\leq p\leq q<+\infty)$ which are needed
in the proof of Theorem \ref{1t4}. Finally, the proof of Theorem
\ref{1t4} is given in \S 6.


\section{Background material}

Our general notation and terminology is standard as can be found, for instance,
in \cite{LT} and \cite{Kechris}. By $\nn=\{1,2,...\}$ we shall denote the
natural numbers. For every infinite subset $L$ of $\nn$ by $[L]$ we denote the
set of all infinite subsets of $L$. If $F$ and $G$ are two non-empty finite
subsets of $\nn$ we write $F<G$ if $\max F<\min G$. Finally, for every set
$A$ by $|A|$ we denote the cardinality of $A$.

\subsection{Trees}

By $\bt$ we denote the set of all finite sequences of natural numbers
while by $\ibt$ we denote the subset of $\bt$ consisting of all strictly
increasing finite sequences (the empty sequence is denoted by $\varnothing$
and is included in both $\bt$ and $\ibt$). We will use the letters $s$ and $t$
to denote elements of $\bt$. By $\sqsubset$ we shall denote the (strict) partial
order on $\bt$ of end-extension. If $\sg\in \bs$ and $k\in\nn$, then we set
$\sg|k=\big(\sg(1),...,\sg(k)\big)$; by convention $\sg|0=\varnothing$.

A \textit{tree on} $\nn$ is a subset of $\bt$ which is closed under initial
segments. By $\tr$ we denote the set of all trees on $\nn$. Hence
\[ S\in\tr \Leftrightarrow \forall s,t\in\bt \ (s\sqsubseteq t
\text{ and } t\in S\Rightarrow s\in S). \]
Notice that $\tr$ is a closed subset of the compact metrizable space $2^{\bt}$.
Also notice that $\ibt\in\tr$. We will reserve the letters $S$ and $R$ to denote
trees. The \textit{body} of a tree $S$ on $\nn$ is defined to be the set
$\{\sigma\in\bs: \sigma|n\in S \ \forall n\in\nn\}$ and is denoted
by $[S]$. A tree $S$ is said to be \textit{well-founded} if $[S]=\varnothing$.
By $\wf$ we denote the set of all well-founded trees on $\nn$.
For every $S\in\wf$ we set
\[ S'=\{s\in S: \exists t\in S \text{ with } s\sqsubset t\}\in\wf.\]
By transfinite recursion, we define the iterated derivatives
$S^{\xi}$ $(\xi<\omega_1)$ of $S$. The \textit{order} $o(S)$ of $S$
is defined to be the least ordinal $\xi$ such that $S^{\xi}=\varnothing$.
By convention, we set $o(S)=\omega_1$ if $S\notin\wf$.

Let $S$ and $R$ be trees on $\nn$. A map $\psi:S\to R$ is said to
be \textit{monotone} if for every $s, s'\in S$ with $s\sqsubset s'$ we have
$\psi(s)\sqsubset \psi(s')$. We notice that if there exists a monotone map
$\psi:S\to R$ and $R$ is well-founded, then $S$ is well-founded and $o(S)\leq o(R)$.

\subsection{Standard Borel spaces}

Let $(X, \Sigma)$ be a standard Borel space; that is, $X$ is a set, $\Sigma$
is a $\sigma$-algebra on $X$ and the measurable space $(X, \Sigma)$ is Borel
isomorphic to the reals. A subset $A$ of $X$ is said to be \textit{analytic}
if there exists a Borel map $f:\nn^\nn\to X$ with $f(\nn^\nn)=A$. A subset
of $X$ is said to be \textit{co-analytic} if its complement is analytic.

A natural, and relevant for our purposes, example of a standard Borel space is
the following. Let $X$ and $Y$ be separable Banach spaces and denote by $\Sigma$
the $\sigma$-algebra on $\llll(X,Y)$ of all Borel subsets of $\llll(X,Y)$ where
$\llll(X,Y)$ is equipped with the strong operator topology. It is well-known and
easy to prove that the measurable space $\big(\llll(X,Y),\Sigma\big)$ is standard
(see \cite[page 80]{Kechris} for more details).

\subsection{Complete co-analytic sets and co-analytic ranks}

Let $B$ be a co-analytic subset of a standard Borel space $X$. The set $B$ is said
to be \textit{co-analytic complete} if for every co-analytic subset $C$ of a
standard Borel space $Y$ there exists a Borel map $f:Y\to X$ such that
$f^{-1}(B)=C$. It is well-known that a complete co-analytic
set is not Borel. We will need the following well-known fact. Its proof is
based on the classical result that the set $\wf$ is co-analytic complete
(see \cite[Theorem 27.1]{Kechris}).
\begin{fact} \label{2f5}
Let $B$ be a co-analytic subset of a standard Borel space $X$. Assume that
there exists a Borel map $h:\tr\to X$ such that $h^{-1}(B)=\wf$.
Then $B$ is complete.
\end{fact}
As above, let $B$ be a co-analytic subset of a standard Borel space $X$.
A map $\varphi:B\to\omega_1$ is said to be a \textit{co-analytic rank on} $B$ if there
exist two binary relations $\leq_\Sigma$ and $\leq_\Pi$ on $X$, which are analytic and
co-analytic respectively, such that for every $y\in B$ we have
\[ \varphi(x)\leq \varphi(y) \Leftrightarrow (x\in B) \text{ and } \varphi(x)\leq\varphi(y)
\Leftrightarrow x\leq_\Sigma y \Leftrightarrow x\leq_\Pi y.\]
A basic property of co-analytic ranks is that they satisfy \textit{boundedness};
that is, if $A$ is an analytic subset of $B$, then $\sup\{\varphi(x):x\in A\}<\omega_1$.
For a proof as well as for a thorough presentation of Rank Theory we refer
to \cite[\S 34]{Kechris}.

We will also need the following.
\begin{fact} \label{2f6}
Let $X$ be a standard Borel space and $\ppp$ be an analytic subset of $X\times\tr$.
Then the set $\ppp^{\sharp}\subseteq X$ defined by
\[ x\in \ppp^{\sharp} \Leftrightarrow \forall S\in\tr \ [(x,S)\in\ppp\Rightarrow S\in\wf]\]
is co-analytic. Moreover, there exists a co-analytic rank $\varphi:\ppp^{\sharp}\to\omega_1$
such that for every $x\in\ppp^{\sharp}$ we have $\sup\{ o(S): S\in\tr \text{ and }
(x,S)\in \ppp\}\leq \varphi(x)$.
\end{fact}
\begin{proof}
First notice that $\ppp^{\sharp}$ is co-analytic since
\[ x\notin\ppp^{\sharp}\Leftrightarrow \exists S\in\tr \text{ such that }
[(x,S)\in\ppp \text{ and } S\notin\wf].\]
The existence of the rank $\varphi$ follows from the parameterized version
of Lusin's Boundedness Theorem for $\wf$. Indeed, by
\cite[page 365]{Kechris} (see also \cite[Theorem 11]{AD}),
there exists a Borel map $f:X\to\tr$ such that
\begin{enumerate}
\item[(a)] $f(x)\notin \wf$ if $x\notin\ppp^{\sharp}$, while
\item[(b)] $f(x)\in\wf$ if $x\in\ppp^{\sharp}$ and $\sup\{o(S):S\in \tr
\text{ and } (x,S)\in\ppp\} \leq o\big(f(x)\big)$.
\end{enumerate}
We set $\varphi(x)=o\big(f(x)\big)$ for every $x\in\ppp^{\sharp}$.
It is easy to check that $\varphi$ is as desired.
\end{proof}

\subsection{Regular families}

Notice that every subset of $\nn$ is naturally identified with an element
of $2^\nn$. We recall the following notions.
\begin{defn} \label{2d7}
Let $\fff$ be a family of finite subsets of $\nn$.
\begin{enumerate}
\item[(1)] The family $\fff$ is said to be \emph{compact} if $\fff$ is
a compact subset of $2^\nn$.
\item[(2)] The family $\fff$ is said to be \emph{hereditary} if for every
$F\in\fff$ and every $G\subseteq F$ we have that $G\in\fff$.
\item[(3)] The family $\fff$ is said to be \emph{spreading} if for every
$F=\{n_1<...<n_k\}\in\fff$ and every $G=\{m_1<...<m_k\}$ with
$n_i\leq m_i$ for all $i\in\{1,...,k\}$ we have that $G\in\fff$.
\item[(4)] The family $\fff$ is said to be \emph{regular} if $\fff$
is compact, hereditary and spreading.
\end{enumerate}
\end{defn}
Regular families are basic combinatorial objects. They have been widely
used in Combinatorics and Functional Analysis (see \cite{AT} for a detailed
exposition). Notice that every regular family $\fff$ is a well-founded tree
on $\nn$, and so, its order $o(\fff)$ can be defined as in \S 2.1.

Let $L=\{l_1<l_2<...\}\in [\nn]$. For every non-empty finite subset $F$ of $\nn$
let $L(F)=\{l_i:i\in F\}$; also let $L(\varnothing)=\varnothing$.
For every regular family $\fff$ we set
\begin{equation} \label{2e3}
\fff[L]=\{F\in\fff: F\subseteq L\} \ \text{ and } \
\fff(L)=\{ L(F): F\in\fff\}.
\end{equation}
We will need the following.
\begin{fact} \label{2f8}
Let $L\in [\nn]$ and $\fff$ be a regular family of finite subsets of $\nn$.
Then the following are satisfied.
\begin{enumerate}
\item[(i)] We have $\fff(L)\subseteq \fff[L]\subseteq \fff$.
\item[(ii)] We have $o\big(\fff(L)\big)=o\big(\fff[L]\big)=o(\fff)$.
\end{enumerate}
\end{fact}
\begin{proof}
Part (i) follows readily from the relevant definitions. To see part (ii),
notice that the map $\fff\ni F\mapsto L(F)\in \fff(L)$ is monotone.
Therefore, $o(\fff)\leq o\big(\fff(L)\big)$ and the result follows.
\end{proof}

\subsection{Schreier families}

The Schreier families $\sss_\xi$ $(1\leq\xi<\omega_1)$ are important
examples of regular families. We recall the definition of the first
two families $\sss_1$ and $\sss_2$ which are more relevant
to the rest of the paper (for more details we refer to \cite{AA,AGR,AT}).
The first Schreier family is defined by
\begin{equation} \label{2e4}
\sss_1=\{ F\subseteq \nn: |F|\leq \min F\}
\end{equation}
while the second one is defined by
\begin{equation} \label{2e5}
\sss_2=\Big\{ \bigcup_{i=1}^n F_i: n\in\nn, \ n\leq \min F_1, \ F_1<...<F_n
\text{ and } F_i\in\sss_1 \ \forall i=1,...,n\Big\}.
\end{equation}
We will need the following facts.
\begin{fact} \label{2f9}
For every $1\leq\xi<\omega_1$ we have $o(\sss_\xi)=\omega^\xi$.
\end{fact}
\begin{fact} \label{2f10}
Let $d\in\nn$ and $N=\{n_1<n_2<...\}\in[\nn]$. Let $1\leq\xi<\omega_1$
and $F\in\sss_\xi$ non-empty. Then we have that
\[ \big\{n_{dk+i-1}:k\in F \text{ and } i\in\{1,...,d\}\big\}\in\sss_\xi.\]
\end{fact}
Fact \ref{2f9} and Fact \ref{2f10} are both proved using transfinite induction.
We leave the details to the interested reader.


\section{A criterion for checking that $\varrho(T)\leq\xi$}

Let $X$ and $Y$ be two Banach spaces with Schauder bases, $T\in\llll(X,Y)$
and $1\leq \xi<\omega_1$. The main result of this section is a simple criterion for checking
that $\varrho(T)\leq \xi$. To state it, we need to introduce the following definition.
\begin{defn} \label{3d11}
Let $X$ and $Y$ be two Banach spaces with normalized bases $(e_n)$ and
$(z_n)$ respectively and $T\in\llll(X,Y)$. We say that two sequences
$(x_n)$ and $(y_n)$, in $X$ and $Y$ respectively, are \emph{$T$-compatible
with respect to $(e_n)$ and $(z_n)$} if the following are satisfied.
\begin{enumerate}
\item[(1)] The sequence $(x_n)$ is a normalized block sequence of $(e_n)$.
\item[(2)] The sequence $(y_n)$ is a seminormalized block sequence of $(z_n)$.
\item[(3)] We have $\|T(x_n)-y_n\|\leq 2^{-n}$ for every $n\in\nn$.
\end{enumerate}
If the bases $(e_n)$ and $(z_n)$ are understood, then we simply say that $(x_n)$
and $(y_n)$ are $T$-compatible.
\end{defn}
We can now state the main result of this section.
\begin{lem} \label{3l12}
Let $X$ and $Y$ be two Banach spaces with normalized bases $(e_n)$ and
$(z_n)$ respectively, $T\in\llll(X,Y)$ and $1\leq \xi<\omega_1$. Then the
following are equivalent.
\begin{enumerate}
\item[(i)] We have $\varrho(T)\leq\xi$.
\item[(ii)] For every pair $(x_n)$ and $(y_n)$ of $T$-compatible sequences
with respect to $(e_n)$ and $(z_n)$ and every $\delta>0$ there exist a non-empty
set $F\in\sss_\xi$ and reals $(a_n)_{n\in F}$ such that
\[ \big\| \sum_{n\in F} a_n x_n \big\|=1  \ \text{ and } \
\big\| \sum_{n\in F} a_n y_n \big\|\leq \delta. \]
\end{enumerate}
\end{lem}
For the proof of Lemma \ref{3l12} we will need the following simple fact.
It was also observed in \cite{ADST}.
\begin{fact} \label{3f13}
Let $X$ and $Y$ be separable Banach spaces, $T\in\llll(X,Y)$ and $1\leq \xi<\omega_1$.
Also let $(x_n)$ be a normalized basic sequence in $X$ and $\ee>0$. Then there exist
a non-empty set $F\in\sss_\xi$ and a norm-one vector $x\in\sspan\{x_n:n\in F\}$
such that $\|T(x)\|\leq\ee$ if and only if there exist a subsequence $(x_{n_k})$
of $(x_n)$, a non-empty set $H\in\sss_\xi$ and a norm-one vector
$x'\in\sspan\{x_{n_k}:k\in H\}$ such that $\|T(x')\|\leq\ee$.
\end{fact}
We proceed to the proof of Lemma \ref{3l12}.
\begin{proof}[Proof of Lemma \ref{3l12}]
It is clear that (i) implies (ii). We work to prove the converse implication.
The arguments are fairly standard, and so, we will be rather sketchy.

Let $(e^*_n)$ and $(z^*_n)$ be the bi-orthogonal functionals associated to
$(e_n)$ and $(z_n)$ respectively. Let $(v_n)$ be a normalized basic sequence
in $X$ and $\ee>0$. We need to find a non-empty set $G\in\sss_\xi$ and a norm-one
vector $v\in\sspan\{v_n: n\in G\}$ such that $\|T(v)\|\leq\ee$. To this end,
by Fact \ref{3f13}, we are allowed to pass to subsequences of $(v_n)$.
Therefore, we may assume that for every $k\in\nn$ the sequences
$\big(e^*_k(v_n)\big)$ and $\big(z^*_k(T(v_n))\big)$ are both convergent.
Let $(d_n)$ be the difference sequence $(v_n)$; that is, $d_n=v_{2n}-v_{2n-1}$
for every $n\in\nn$. Notice that
\begin{enumerate}
\item[(a)] the sequence $(d_n)$ is seminormalized,
\item[(b)] $e^*_k(d_n)\to 0$ for every $k\in\nn$ and
\item[(c)] $z_k^*\big(T(d_n)\big)\to 0$ for every $k\in\nn$.
\end{enumerate}
By (a) and (b) and by passing to a subsequence of $(v_n)$, it is possible
to find a seminormalized block sequence $(b^0_n)$ of $(e_n)$ such that
$\|b^0_n-d_n\|\leq 2^{-n}$ for every $n\in\nn$. Now we distinguish
the following (mutually exclusive) cases.
\medskip

\noindent \textsc{Case 1:} \textit{There exists a subsequence of $\big(T(d_n)\big)$
which is norm convergent to $0$.} In this case it is easy to see that there exist
$G\in\sss_\xi$ with $|G|=2$ and a norm-one vector $v\in\sspan\{v_n:n\in G\}$
such that $\|T(v)\|\leq\ee$.
\medskip

\noindent \textsc{Case 2:} \textit{There exists a subsequence of $\big(T(d_n)\big)$
which is seminormalized.} In this case, by (c) above and by passing
to a further subsequence of $(v_n)$, we may find a seminormalized block sequence
$(b^1_n)$ of $(z_n)$ such that $\|b^1_n-T(d_n)\|\leq 2^{-n}$ for every $n\in\nn$.
Summing up, we see that it is possible to select an infinite subset
$N=\{n_1<n_2<...\}$ of $\nn$ such that, setting
\[ w_k=v_{n_{2k+1}}-v_{n_{2k}}\]
for every $k\in\nn$, the sequences $(w_k)$ and $\big(T(w_k)\big)$
are both seminormalized and ``almost block". Hence, using our hypotheses,
we may find a non-empty set $F\in\sss_\xi$ and a norm-one vector $v\in\sspan\{w_k:k\in F\}$
such that $\|T(v)\|\leq\ee$. By Fact \ref{2f10}, we see that
$G:=\big\{n_{2k+i-1}:k\in F \text{ and } i\in\{1,2\}\big\}\in\sss_\xi$
and the result follows.
\end{proof}


\section{Proof of Theorem \ref{1t2}}

Let $X$ and $Y$ be separable Banach spaces. Let $\bbb$ be the subset of
$X^\nn$ defined by
\[ (x_n)\in\bbb\Leftrightarrow (x_n) \text{ is a normalized basic sequence}.\]
It is easy to see that $\bbb$ is an $F_\sg$ subset of $X^\nn$. Hence, the set $\bbb$
equipped with the relative Borel $\sg$-algebra of $X^\nn$ is a standard Borel space
(see \cite{Kechris}).

For every $T\in\llll(X,Y)$, every $(x_n)\in\bbb$ and every $m\in\nn$ we introduce
a tree $S(T,(x_n),m)$ on $\nn$ defined by the rule
\begin{eqnarray} \label{4e6}
s\in S(T,(x_n),m) & \Leftrightarrow & \text{either } s=\varnothing \text{ or }
s=(n_1<...<n_k)\in\ibt \text{ and} \\
\nonumber & & \forall d\in\nn \text{ with } d\leq k, \ \forall (l_1<...<l_d)\in\ibt \text{ with} \\
\nonumber & & n_i\leq l_i \text{ for every } i\in\{1,...,d\} \text{ and }
\forall a_1,...,a_d\in\qq \\
\nonumber & & \text{we have } \big\| T\big( \sum_{i=1}^d a_i x_{l_i}\big) \big\|
\geq \frac{1}{m} \big\| \sum_{i=1}^d a_i x_{l_i} \big\|.
\end{eqnarray}
We notice the following simple facts. The proofs are left to the reader.
\begin{fact} \label{4f14}
The map $\llll(X,Y)\times\bbb\times\nn\ni (T,(x_n),m)\mapsto S(T,(x_n),m)\in\tr$
is Borel.
\end{fact}
\begin{fact} \label{4f15}
Let $T\in\llll(X,Y)$. If $T$ is not strictly singular, then there exist $(x_n)\in\bbb$
and $m\in\nn$ such that the tree $S(T,(x_n),m)$ is not well-founded.
\end{fact}
We proceed to analyze the above defined trees when the operator $T$ is strictly
singular.
\begin{claim} \label{4c16}
Let $T\in\sss\sss(X,Y)$ with $\varrho(T)=\xi$. Also let $(x_n)\in\bbb$ and $m\in\nn$.
Then the tree $S(T,(x_n),m)$ is a regular family. Moreover,
\begin{equation} \label{4e7}
o\big(S(T,(x_n),m)\big)\leq \omega^{\xi+1}.
\end{equation}
\end{claim}
\begin{proof}
For notational simplicity, let us denote by $\fff$ the tree $S(T,(x_n),m)$.
It is clear from the definition that $\fff$ is a hereditary and spreading
family of finite subsets of $\nn$. It is easy to see that $\fff$ is in addition
well-founded. This implies that $\fff$ is compact in $2^\nn$. Hence, $\fff$
is a regular family.

We work now to prove that $o(\fff)\leq \omega^{\xi+1}$. We argue by contradiction.
So, assume that $o(\fff)>\omega^{\xi+1}$. A result of I. Gasparis \cite{Ga}
asserts that if $\gggg$ and $\hhh$ are two hereditary families of finite subsets
of $\nn$, then there exists $L\in[\nn]$ such that either $\gggg[L]\subseteq \hhh$
or $\hhh[L]\subseteq \gggg$. Applying this dichotomy to the families $\fff$ and
$\sss_{\xi+1}$ we find $L=\{l_1<l_2<...\}\in[\nn]$ such that either
$\fff[L]\subseteq\sss_{\xi+1}$ or $\sss_{\xi+1}[L]\subseteq \fff$. We claim that
the first case is impossible. Indeed, assume on the contrary that
$\fff[L]\subseteq \sss_{\xi+1}$. By Fact \ref{2f8}(ii) and Fact \ref{2f9}, we see that
\[ \omega^{\xi+1}< o(\fff)=o\big(\fff[L]\big)\leq o(\sss_{\xi+1})=\omega^{\xi+1}\]
which is clearly impossible. Hence, $\sss_{\xi+1}[L]\subseteq \fff$.

Introduce now the sequence $(z_n)$ defined by the rule that $z_n=x_{l_n}$ for every
$n\in\nn$. Clearly $(z_n)$ is a normalized basic sequence in $X$. Let $F\in\sss_{\xi+1}$
be arbitrary and non-empty. The family $\sss_{\xi+1}$ is regular.
Hence, by Fact \ref{2f8}(i), we get that
\[ L(F)=\{l_n:n\in F\}\in \sss_{\xi+1}(L)\subseteq \sss_{\xi+1}[L]\subseteq \fff.\]
By the definition of $\fff$ and the continuity of the operator $T$, we see that for
every choice $(a_n)_{n\in F}$ of reals we have
\[ \big\| T\big( \sum_{n\in F} a_n z_n \big) \big\|=
\big\| T\big( \sum_{n\in F} a_n x_{l_n}\big) \big\| \geq \frac{1}{m}
\big\| \sum_{n\in F} a_n x_{l_n} \big\| = \frac{1}{m}
\big\| \sum_{n\in F} a_n z_n \big\|. \]
In other words, we conclude that for every non-empty set $F\in\sss_{\xi+1}$ and every
norm-one vector $z\in\sspan\{z_n:n\in F\}$ we have $\|T(z)\|\geq m^{-1}$. This implies
that $T$ is not $\sss_{\xi+1}$-strictly singular. By \cite[Proposition 2.4]{ADST},
the operator $T$ is not $\sss_\zeta$-strictly singular for every $1\leq\zeta\leq\xi+1$,
and so, $\varrho(T)>\xi+1$. This is a contradiction. Therefore, $o(\fff)\leq \omega^{\xi+1}$
and the proof is completed.
\end{proof}
As a consequence we get the following result which shows that the family of trees
$\{S(T,(x_n),m):(x_n)\in \bbb \text{ and } m\in\nn\}$ can be used to compute the
ordinal $\varrho(T)$ quite accurately.
\begin{cor} \label{4c17}
Let $T\in\sss\sss(X,Y)$ with $\varrho(T)=\xi$. Then
\begin{equation} \label{4e8}
\sup\{\omega^{\zeta}:\zeta<\xi\} \leq \sup\big\{ o\big(S(T,(x_n),m)\big): (x_n)\in\bbb \text{ and }
m\in \nn\big\} \leq \omega^{\xi+1}.
\end{equation}
\end{cor}
\begin{proof}
The second inequality follows immediately by Claim \ref{4c16}. We work to prove the
first inequality. Clearly we may assume that $\xi>1$. Let $\zeta$ be an arbitrary
countable ordinal with $1\leq\zeta<\xi$. Since $\varrho(T)>\zeta$, the operator
$T$ is not $\sss_\zeta$-strictly singular. Therefore, we may find $(x_n)\in\bbb$
and $\ee>0$ such that for every non-empty set $F\in\sss_\zeta$ and every $x\in\sspan
\{x_n:n\in F\}$ we have $\|T(x)\|\geq \ee\|x\|$. We select $m\in\nn$ such that
$\ee\geq m^{-1}$. The family $\sss_\zeta$ is spreading and hereditary. Hence, by the
definition of the tree $S(T,(x_n),m)$, we see that $F\in S(T,(x_n),m)$ for every $F\in\sss_\zeta$.
In particular, the identity map $\mathrm{Id}:\sss_\zeta\to S(T,(x_n),m)$ is a well-defined
monotone map. Therefore, by Fact \ref{2f9}, we see that
\[ \omega^{\zeta}=o(\sss_\zeta)\leq o\big(S(T,(x_n),m)\big) \]
and the result follows.
\end{proof}
Now, define $\ppp\subseteq \llll(X,Y)\times \tr$ by the rule
\begin{equation} \label{4e9}
(T,R)\in\ppp \Leftrightarrow \exists (x_n)\in\bbb \text{ and }
\exists m\in\nn \text{ such that } R=S(T,(x_n),m).
\end{equation}
By Fact \ref{4f14}, we see that the set $\ppp$ is analytic. As in
Fact \ref{2f6}, let $\ppp^{\sharp}\subseteq\llll(X,Y)$ be defined by
\[ T\in\ppp^{\sharp} \Leftrightarrow \forall R\in \tr \ [(T,R)\in\ppp
\Rightarrow R\in\wf].\]
By Fact \ref{4f15} and Claim \ref{4c16}, we get that $\ppp^{\sharp}=
\sss\sss(X,Y)$. Let $\varphi:\ppp^{\sharp}\to\omega_1$ be the
co-analytic rank on $\ppp^{\sharp}$ obtained in Fact \ref{2f6}.

We define
\begin{equation} \label{4e10}
\rs(T)=\varphi(T)+1
\end{equation}
for every $T\in\sss\sss(X,Y)$ and we claim that $\rs$ is the desired rank.
Clearly $\rs$ is a co-analytic rank on $\sss\sss(X,Y)$. It remains to check
that $\varrho(T)\leq\rs(T)$ for every $T\in\sss\sss(X,Y)$. To this end,
fix $T\in\sss\sss(X,Y)$. By Fact \ref{2f6}, we have that
\begin{equation} \label{4e11}
\sup\{ o(R): (T,R)\in\ppp\}\leq \varphi(T)
\end{equation}
while by the definition of the set $\ppp$ we get that
\begin{equation} \label{4e12}
\sup\big\{ o\big(S(T,(x_n),m)\big): (x_n)\in\bbb \text{ and } m\in\nn\big\}=
\sup\{ o(R): (T,R)\in\ppp\}.
\end{equation}
Finally, notice that
\begin{equation} \label{4e13}
\xi\leq \sup\{\omega^{\zeta}:\zeta<\xi\}+1
\end{equation}
for every countable ordinal $\xi$. Combining inequalities (\ref{4e8}), (\ref{4e11}),
(\ref{4e12}) and (\ref{4e13}) we conclude that $\varrho(T)\leq\rs(T)$ as desired.

Finally, to see that the rank $\varrho$ satisfies boundedness, let $\aaa$ be an
analytic subset of $\sss\sss(X,Y)$. The rank $\rs$ is a co-analytic rank. Therefore,
there exists a countable ordinal $\xi$ such that $\rs(T)\leq\xi$ for every
$T\in\aaa$. Hence, $\sup\{\varrho(T):T\in\aaa\}\leq\xi<\omega_1$.

The proof of Theorem \ref{1t2} is completed.


\section{The spaces $Z_{p,q}$ $(1\leq p\leq q<+\infty)$}

This section contains some results which are needed for the proof of Theorem \ref{1t4}
stated in the introduction. It is organized as follows. In \S 5.1 we introduce
some pieces of notation. In \S 5.2 we define the space $Z_{p,q}$ $(1\leq p\leq q<+\infty)$
and we gather some of its basic properties. Finally, in \S 5.3 we present a
result concerning a class of sequences in $Z_{p,q}$ which we call
``asymptotically sparse".

\subsection{Notation}

For the rest of the paper we fix a bijection $\chi:\bt\to \nn$ satisfying $\chi(s)<\chi(t)$
for every $s,t\in\bt$ with $s\sqsubset t$.

Let $s,t\in\bt$. The nodes $s$ and $t$ are said to be \textit{comparable} if either
$s\sqsubseteq t$ or $t\sqsubseteq s$; otherwise $s$ and $t$ are said to be \textit{incomparable}.
A subset of $\bt$ consisting of pairwise comparable nodes is said to be a \textit{chain}
while a subset of $\bt$ consisting of pairwise incomparable nodes is said to be an
\textit{antichain}. A \textit{branch} of $\bt$ is a maximal chain of $\bt$. The branches
of $\bt$ are naturally identified with the elements of $\nn^\nn$; indeed, a subset $A$
of $\bt$ is a branch if and only if there exists $\sg\in\nn^\nn$ (unique) such that
$A=\{\sg|n:n\geq 0\}$. Two subsets $A$ and $B$ of $\bt$ are said to
be \textit{incomparable} if for every $s\in A$ and every $t\in B$ the nodes $s$ and $t$
are incomparable. A \textit{segment} $\seg$ of $\bt$ is a chain of $\bt$ satisfying
\begin{equation} \label{5e14}
\forall s, t, s'\in\bt \ (s\sqsubseteq t\sqsubseteq s' \text{ and } s, s'\in\seg
\Rightarrow t\in\seg).
\end{equation}
If $\seg$ is a segment of $\bt$, then by $\min(\seg)$ we denote the
$\sqsubseteq$-minimal node of $\seg$. Notice that two segments $\seg$ and $\seg'$
are incomparable if and only if the nodes $\min(\seg)$ and $\min(\seg')$
are incomparable. If $\sg$ is a branch of $\bt$ and $k\geq 0$, then the set $\{\sg|n:n\geq k\}$
is said to be a \textit{final segment} of $\sg$ while the set $\{\sg|n:n\leq k\}$ is
said to be an \textit{initial segment} of $\sg$.

\subsection{Definitions and basic properties}

We start with the following.
\begin{defn} \label{5d18}
Let $1\leq p\leq q<+\infty$. We define $Z_{p,q}$ to be the completion of $c_{00}(\bt)$
equipped with the norm
\begin{equation} \label{5e15}
\|z\|_{Z_{p,q}}=\sup\Big\{ \Big( \sum_{i=1}^d \big( \sum_{t\in\seg_i} |z(t)|^p\big)^{q/p} \Big)^{1/q} \Big\}
\end{equation}
where the above supremum is taken over all families $(\seg_i)_{i=1}^d$ of pairwise
incomparable non-empty segments of $\bt$.
\end{defn}
The space $Z_{p,q}$ is a variant of James tree space $JT$ \cite{Ja}.
We notice that spaces of this form have found significant applications and have been extensively
studied by several authors (see, for instance, \cite{AD,Bos,Bou,D,DL,Fe}). We gather,
below, some elementary properties of the space $Z_{p,q}$.

Let $\{z_t:t\in\bt\}$ be the standard Hamel basis of $c_{00}(\bt)$ and
$(t_n)$ be the enumeration of $\bt$ according to the bijection $\chi$ (see \S 5.1).
The sequence $(z_{t_n})$ defines an $1$-unconditional basis of $Z_{p,q}$.
For every node $t$ of $\bt$ by $z^*_t$ we shall denote the bi-orthogonal functional
associated to $z_t$. For every vector $z$ in $Z_{p,q}$ the \textit{support}
$\supp(z)$ of $z$ is defined to be the set $\{t\in\bt: z^*_t(z)\neq 0\}$.

For every $A\subseteq \bt$ non-empty let
\begin{equation} \label{5e16}
Z^A_{p,q}=\overline{\mathrm{span}}\{z_t:t\in A\}.
\end{equation}
The subspace $Z^A_{p,q}$ of $Z_{p,q}$ is complemented via the natural projection
\begin{equation} \label{5e17}
P_A:Z_{p,q}\to Z^A_{p,q}.
\end{equation}
Notice that $\|P_A\|=1$. Observe that for every non-empty chain $\mathfrak{c}$ of $\bt$
and every vector $z$ in $Z_{p,q}$ we have
\begin{equation} \label{5e18}
\|P_{\mathfrak{c}}(z)\|=\Big( \sum_{t\in\mathfrak{c}} |z^*_t(z)|^p \Big)^{1/p}.
\end{equation}
In particular, for every branch $\sg$ of $\bt$ the subspace $Z^\sg_{p,q}$ of $Z_{p,q}$
is isometric to $\ell_p$ and complemented via the norm-one
projection $P_\sg:Z_{p,q}\to Z^\sg_{p,q}$.

Let $X$ and $E$ be infinite-dimensional Banach spaces. Recall that the space $X$
is said to be \textit{hereditarily} $E$ if every infinite-dimensional subspace of
$X$ contains an isomorphic copy of $E$. We will need the following easy
(and essentially known) fact concerning the structure of the space $Z^S_{p,q}$ when
$S$ is a well-founded tree. The proof is sketched for completeness.
\begin{fact} \label{5f19}
Let $1\leq p\leq q<+\infty$ and $S\in\wf$. Then the space $Z^S_{p,q}$ is either
finite-dimensional or hereditarily $\ell_q$.
\end{fact}
\begin{proof}
We proceed by induction on the order of the tree $S$. If $o(S)=1$, then the
space $Z^S_{p,q}$ is one-dimensional. Let $S\in\wf$ with $o(S)>1$ and assume
that the result has been proved for every $R\in\wf$ with $o(R)<o(S)$. We set
\[L_S=\{n\in\nn: (n)\in S\}. \]
For every $n\in L_S$ let $S_n=\{t\in\bt:n^{\con}t\in S\}$
and notice that $S_n\in\wf$ and $o(S_n)<o(S)$. Therefore, by our
induction hypothesis, the space $Z^{S_n}_{p,q}$ is
either finite-dimensional or hereditarily $\ell_q$. Noticing that
the space $Z^S_{p,q}$ is isomorphic to the space
\[ \rr\oplus \Big( \sum_{n\in L_S} \oplus Z^{S_n}_{p,q}\Big)_{\ell_q} \]
the result follows.
\end{proof}

\subsection{Asymptotically sparse sequences in $Z_{p,q}$}

We start by introducing the following definition.
\begin{defn} \label{5d20}
Let $1\leq p\leq q<+\infty$. We say that a bounded block sequence $(y_n)$ in $Z_{p,q}$
is \emph{asymptotically sparse} if for every $k\in\nn$ and every $\sg\in\nn^\nn$
we have
\begin{equation} \label{5e19}
|\{n\geq k: \|P_\sg(y_n)\|\geq 2^{-k}\}|\leq 1.
\end{equation}
\end{defn}
Notice that if $(y_n)$ is an asymptotically sparse sequence,
then $\|P_\sg(y_n)\|\to 0$ for every $\sg\in\nn^\nn$.
The main result of this subsection asserts that (essentially) the converse
is also true. Precisely, we have the following.
\begin{lem} \label{5l21}
Let $1\leq p\leq q<+\infty$ and $(y_n)$ be a bounded block sequence in $Z_{p,q}$
such that $\|P_\sg(y_n)\|\to 0$ for every $\sg\in\nn^\nn$. Then $(y_n)$
has an asymptotically sparse subsequence.
\end{lem}
Lemma \ref{5l21} is a Ramsey-theoretical result and the arguments in its proof
can be traced in the work of I. Amemiya and T. Ito \cite{AI} concerning the structure
of normalized weakly null sequences in the space $JT$. We proceed to the proof.
\begin{proof}[Proof of Lemma \ref{5l21}]
Let $1\leq p\leq q<+\infty$ and fix a bounded block sequence $(y_n)$ in $Z_{p,q}$
such that $\|P(y_n)\|\to 0$ for every $\sg\in\nn^\nn$. We select $C>0$ such that
$\|y_n\|\leq C$ for every $n\in\nn$.
\begin{claim} \label{5c22}
For every $\theta>0$ and every $M\in[\nn]$ there exists $L\in[M]$ such that
for every $\sg\in \nn^\nn$ we have $|\{n\in L:\|P_\sg(y_n)\|\geq\theta\}|\leq 1$.
\end{claim}
Granting Claim \ref{5c22} the proof of Lemma \ref{5l21} is completed. Indeed,
by repeated applications of Claim \ref{5c22}, it is possible to find a sequence
$(L_k)$ of infinite subsets of $\nn$ such that for every $k\in\nn$ the following
are satisfied.
\begin{enumerate}
\item[(a)] $\min L_k<\min L_{k+1}$.
\item[(b)] $L_{k+1}\subseteq L_k$.
\item[(c)] $|\{n\in L_k:\|P_\sg(y_n)\|\geq 2^{-k}\}|\leq 1$ for every $\sg\in\nn^\nn$.
\end{enumerate}
Introduce the sequence $(w_k)$ in $Z_{p,q}$ defined by $w_k=y_{\min L_k}$ for every
$k\in\nn$. By (a) above, we see that $(w_k)$ is a subsequence of $(y_n)$ while,
by (b) and (c), the sequence $(w_k)$ is asymptotically sparse.

It remains to prove Claim \ref{5c22}. We will argue by contradiction.
So, assume that there exist $\theta>0$ and $M\in[\nn]$ such that for every $L\in [M]$
there exist $m,k\in L$ with $m<k$ and $\sg\in\nn^\nn$ such that $\|P_\sg(y_m)\|\geq\theta$
and $\|P_\sg(y_k)\|\geq\theta$. Therefore, applying the classical Ramsey Theorem
\cite{Ra} and by passing to a subsequence of $(y_n)$, we may assume that for every
$m,k\in\nn$ with $m<k$ there exists $\sg_{m,k}\in\nn^\nn$ such that
$\|P_{\sg_{m,k}}(y_m)\|\geq\theta$ and $\|P_{\sg_{m,k}}(y_k)\|\geq\theta$.

Fix $k\in\nn$ with $k\geq 2$. For every $m\in\nn$ with $m<k$ let $\seg^-_{m,k}$ be the
maximal initial segment of $\sg_{m,k}$ which is disjoint from $\supp(y_k)$. As the sequence
$(y_n)$ is block, we see that $\|P_{\seg^-_{m,k}}(y_m)\|\geq\theta$. Let
$\seg^+_{m,k}=\sg_{m,k}\setminus \seg^-_{m,k}$ and notice that $\seg^+_{m,k}$
is a final segment of $\sg_{m,k}$ and that $\|P_{\seg^+_{m,k}}(y_k)\|\geq\theta$.
Moreover, $\min(\seg^+_{m,k})\in\supp(y_k)$. For every $r>0$ let
$\lceil r \rceil$ be the least $k\in\nn$ such that $r\leq k$. Now we observe that
\begin{equation} \label{5e20}
|\{\seg^-_{m,k}:m<k\}|\leq \lceil C^q/\theta^q\rceil.
\end{equation}
Indeed, let $\seg_1,...,\seg_d$ be an enumeration of the set $\{\seg^-_{m,k}:m<k\}$.
Then for every $i\in\{1,...,d\}$ there exists $m_i<k$ such that $\seg_i=\seg^-_{m_i,k}$.
Since the segments $(\seg^-_{m_i,k})_{i=1}^d$ are mutually different, the final segments
$(\seg^+_{m_i,k})_{i=1}^d$ are pairwise incomparable. To see this assume, towards a
contradiction, that there exist $i,j\in\{1,...,d\}$ such that
$\min(\seg^+_{m_i,k})$ is a proper initial segment of $\min(\seg^+_{m_j,k})$.
As $\min(\seg^+_{m_i,k})\in\supp(y_k)$ we get that
$\min(\seg^+_{m_i,k})\in\supp(y_k)\cap \seg^-_{m_j,k}$ contradicting
the fact that $\seg^-_{m_j,k}$ is disjoint from $\supp(y_k)$. Therefore,
the final segments $(\seg^+_{m_i,k})_{i=1}^d$ are pairwise incomparable,
and so, $C\geq \|y_k\|\geq \theta\cdot d^{1/q}$ which gives the desired estimate.

Set $D=\lceil C^q/\theta^q\rceil$. By the previous discussion, for every $k\in\nn$ with
$k\geq 2$ there exists a family $\{\seg_{i,k}:i=1,...,D\}$ of initial segments of $\bt$
such that for every $m\in\nn$ with $m<k$ there exists $i\in\{1,...,D\}$ such that
$\|P_{\seg_{i,k}}(y_m)\|\geq\theta$. The space $2^{\bt}$ is compact. Therefore,
by passing to subsequences, we may find a family $\{\seg_1,...,\seg_D\}$ of initial
segments of $\bt$ such that $\seg_{i,k}\to\seg_i$ in $2^{\bt}$ for every $i\in\{1,...,D\}$.

Let $m,k\in\nn$ with $m<k$ and $i\in \{1,...,D\}$. Let us say that $k$ is $i$-\textit{good
for} $m$ if $\|P_{\seg_{i,k}}(y_m)\|\geq\theta$. Notice that for every $m\in\nn$ there exists
$i\in\{1,...,D\}$ such that the set $H^i_m=\{k>m:k \text{ is } i-\text{good for } m\}$ is
infinite. Hence, there exist $j\in\{1,...,D\}$ and $N\in[\nn]$ such that $H^j_m$ is infinite
for every $m\in N$. We select $\tau\in\nn^\nn$ such that $\seg_{j}$ is an initial
segment of $\tau$. Since $\seg_{j,k}\to\seg_{j}$ in $2^{\bt}$ and
$\|P_{\seg_{j,k}}(y_m)\|\geq\theta$ for every $m\in N$ and every $k\in H^j_m$, we get that
\[ \limsup_{m\in N} \|P_\tau(y_m)\|\geq \limsup_{m\in N} \|P_{\seg_j}(y_m)\|=
\limsup_{m\in N} \lim_{k} \|P_{\seg_{j,k}}(y_m)\|\geq\theta.\]
This is clearly a contradiction. The proof of Lemma \ref{5l21} is completed.
\end{proof}


\section{Proof of Theorem \ref{1t4}}

Let $1\leq p<+\infty$. We set
\begin{equation} \label{6e21}
q=2p
\end{equation}
and we define $Y_p$ to be the space $Z_{p,q}$. By \S 5.2, the space $Y_p$
has a normalized $1$-unconditional basis $(z_{t_n})$. Let $(e_n)$ be the
standard unit vector basis of $\ell_p$. By $\irm:\ell_p\to Y_p$ we shall
denote the unique norm-one operator satisfying
\begin{equation} \label{6e22}
\irm(e_n)=z_{t_n}
\end{equation}
for every $n\in\nn$. We proceed to show that the space $Y_p$ is the desired one.

\subsection{The set $\sss\sss(\ell_p,Y_p)$ is a complete co-analytic subset
of $\llll(\ell_p,Y_p)$}

As we have already mentioned in the introduction, the set $\sss\sss(X,Y)$
is a co-analytic subset of $\llll(X,Y)$ for every pair $X$ and $Y$ of separable
Banach space. Hence, what remains is to show that the set $\sss\sss(\ell_p,Y_p)$
is actually complete. By Fact \ref{2f5}, it is enough to find a Borel
map $H:\tr\to\llll(\ell_p,Y_p)$ such that for every $S\in\tr$ we have
\[ S\in\wf \Leftrightarrow H(S)\in\sss\sss(\ell_p,Y_p).\]
To this end, let $S\in\tr$ be arbitrary. Let $Z^S_{p,q}$ be the subspace of
$Y_p$ defined in (\ref{5e16}) and $P_S:Y_p\to Z^S_{p,q}$ be the natural
norm-one projection. We define
\begin{equation} \label{6e23}
H(S)= P_S \circ \irm\in\llll(\ell_p,Y_p).
\end{equation}
Notice that $\|H(S)\|=1$.
\begin{claim} \label{6c23}
The map $H:\tr\to\llll(\ell_p,Y_p)$ is continuous when $\llll(\ell_p,Y_p)$
is equipped with the strong operator topology.
\end{claim}
\begin{proof}
Let $(S_n)$ be a sequence in $\tr$ and $S\in\tr$ such that $S_n\to S$.
Notice that for every $s\in\bt$ we have $s\in S$ if and only if $s\in S_n$
for all $n\in\nn$ large enough. Let $x\in \ell_p$ be arbitrary and set
$y=\irm(x)$. It follows from the above remarks that for every $r>0$
there exists $k\in\nn$ such that $\|P_S(y)-P_{S_n}(y)\|\leq r$ for every
$n\in\nn$ with $n\geq k$ and the result follows.
\end{proof}
\begin{claim} \label{6c24} Let $S\in\tr$. Then $S\in\wf$ if and only if
$H(S)\in\sss\sss(\ell_p,Y_p)$.
\end{claim}
\begin{proof}
First assume that $S\in\wf$. Notice that the operator $H(S)$ maps $\ell_p$
onto $Z^S_{p,q}$. By Fact \ref{5f19}, the space $Z^S_{p,q}$ is either
finite-dimensional or hereditarily $\ell_q$. Since $p\neq q$, the operator
$H(S)$ is strictly singular.

Now assume that $S\notin \wf$ and let $\sg\in [S]$. Let $\chi:\bt\to\nn$
be the bijection described in \S 5.1 and for every $k\in\nn$ set
$n_k=\chi\big(\sg(k)\big)$. By the properties of $\chi$, we see that
$n_k<n_{k+1}$ for every $k\in\nn$. Let $E$ be the subspace of $\ell_p$
spanned by the subsequence $(e_{n_k})$ of the basis $(e_n)$. We claim
that the operator $H(S)$ restricted on $E$ is an isometric embedding.
Indeed, let $d\in\nn$ and $a_1,...,a_d\in\rr$ and notice that
\begin{eqnarray*}
\big\| \sum_{k=1}^d a_k e_{n_k} \big\|_{\ell_p} & = & \Big( \sum_{k=1}^d |a_k|^p\Big)^{1/p}
= \big\| \sum_{k=1}^d a_k z_{\sg(k)} \big\|_{Y_p} \\
& = & \big\| (P_S\circ \irm)\Big(\sum_{k=1}^d a_k e_{n_k}\Big)\big\|_{Y_p}.
\end{eqnarray*}
The claim is proved.
\end{proof}
By Fact \ref{2f5}, Claim \ref{6c23} and Claim \ref{6c24}, we conclude that
$\sss\sss(\ell_p,Y_p)$ is a complete co-analytic subset of $\llll(\ell_p,Y_p)$.

\subsection{For every $T\in\sss\sss(\ell_p,Y_p)$ we have $\varrho(T)\leq 2$}

Let us fix a strictly singular operator $T:\ell_p\to Y_p$. We need to prove that
$\varrho(T)\leq 2$. To this end, we may assume that
\begin{enumerate}
\item[(P1)] $\|T\|=1$.
\end{enumerate}
By Lemma \ref{3l12}, it is enough to show that for every pair $(x_n)$ and
$(y_n)$ of $T$-compatible sequences (with respect to the bases $(e_n)$ and
$(z_{t_n})$ of $\ell_p$ and $Y_p$  respectively) and for every $\delta>0$
there exist a non-empty set $F\in\sss_2$ and reals $(a_n)_{n\in F}$ such that
\[ \big\| \sum_{n\in F} a_n x_n\big\|_{\ell_p}=1 \ \text{ and } \
\big\| \sum_{n\in F} a_n y_n\big\|_{Y_p}\leq\delta.\]
So, fix a pair $(x_n)$ and $(y_n)$ of $T$-compatible sequences
and $\delta>0$. By Definition \ref{3d11} and (P1) above, we see that
the following are satisfied.
\begin{enumerate}
\item[(P2)] The sequence $(x_n)$ is $1$-equivalent to the standard unit
vector basis of $\ell_p$.
\item[(P3)] The sequence $(y_n)$ is block and satisfies $\|y_n\|\leq 2$
for every $n\in\nn$.
\end{enumerate}
We are going to refine the sequences $(x_n)$ and $(y_n)$ in order to achieve
further properties. Observe that we are allowed to do so since the family
$\sss_2$ is spreading. First we notice that, by Definition \ref{3d11}
and by passing to common subsequences of $(x_n)$ and $(y_n)$ if necessary,
we may find a constant $\Theta\geq 1$ such that
\begin{enumerate}
\item[(P4)] the sequences $(y_n)$ and $\big(T(x_n)\big)$ are
$\Theta$-equivalent.
\end{enumerate}
Now we make the following simple (but crucial) observation.
\begin{lem} \label{6l25}
For every $\sg\in\nn^\nn$ we have $\|P_\sg(y_n)\|\to 0$.
\end{lem}
\begin{proof}
Assume, towards a contradiction, that there exist $\sg\in\nn^\nn$,
a constant $\theta>0$ and $L=\{l_1<l_2<...\}\in[\nn]$ such that
$\|P_\sg(y_{l_n})\|\geq\theta$ for every $n\in\nn$. Since the sequence
$(y_n)$ is block and $\|P_\sg\|=1$, this implies that for every $d\in\nn$
and every $a_1,...,a_d\in\rr$ we have
\begin{equation} \label{6e24}
\theta \Big( \sum_{n=1}^d |a_n|^p\Big)^{1/p} \leq
\big\| P_\sg\Big(\sum_{n=1}^d a_n y_{l_n}\Big)\big\|_{Y_p} \leq
\big\| \sum_{n=1}^d a_n y_{l_n}\big\|_{Y_p}.
\end{equation}
Let $E$ be the subspace of $\ell_p$ spanned by the subsequence $(x_{l_n})$ of $(x_n)$.
We claim that the operator $T$ restricted on $E$ is an isomorphic embedding.
Indeed, let $d\in\nn$ and $a_1,...,a_d\in\rr$ and notice that
\begin{eqnarray*}
\theta \big\|\sum_{n=1}^d a_n x_{l_n}\big\|_{\ell_p}
& \stackrel{\text{(P2)}}{=} & \theta \Big( \sum_{n=1}^d |a_n|^p\Big)^{1/p}
\stackrel{\text{(\ref{6e24})}}{\leq} \big\| \sum_{n=1}^d a_n y_{l_n}\big\|_{Y_p} \\
& \stackrel{\text{(P4)}}{\leq} & \Theta \big\|\sum_{n=1}^d a_n T(x_{l_n})\big\|_{Y_p}
=  \Theta \big\|T\Big(\sum_{n=1}^d a_n x_{l_n}\Big)\big\|_{Y_p}.
\end{eqnarray*}
Therefore, for every $x\in E$ with $\|x\|=1$ we have
$\|T(x)\|\geq \theta\cdot \Theta^{-1}$. This is clearly a contradiction
and the proof is completed.
\end{proof}
By (P3) above, Lemma \ref{6l25} and Lemma \ref{5l21} and by passing
to further common subsequences of $(x_n)$ and $(y_n)$, we may additionally
assume that
\begin{enumerate}
\item[(P5)] the sequence $(y_n)$ is asymptotically sparse.
\end{enumerate}
We fix $N\in\nn$ with $N\geq 2$ and such that
\begin{equation} \label{6e25}
N^{1/q-1/p} \leq \delta \cdot (2\Theta)^{-1}.
\end{equation}
Such a natural number can be found since $q=2p$ and $p\geq 1$.
Recursively, for every $i\in\{1,...,N\}$ we will select
\begin{enumerate}
\item[(a)] a natural number $k_i$,
\item[(b)] a positive real $\ee_i$ and
\item[(c)] a non-empty finite subset $F_i$ of $\nn$
\end{enumerate}
such that, setting $\mu_1=1$ and
\begin{equation} \label{6e26}
\mu_i=\sum_{m=1}^{i-1} \Big( \sum_{n\in F_m} |\supp(y_n)|\Big)^{1/q}
\end{equation}
for every $i\in\{2,...,N\}$, the following conditions are satisfied.
\begin{enumerate}
\item[(C1)] $F_1<...<F_N$ and $N\leq \min F_1$.
\item[(C2)] $|F_i|=k_i$ and $k_i\leq \min F_i$ for every $i\in\{1,...,N\}$.
\item[(C3)] $|\{n\in F_i: \|P_\sg(y_n)\|\geq\ee_i\}|\leq 1$ for every
$\sg\in\nn^\nn$ and every $i\in\{1,...,N\}$.
\item[(C4)] $(k_i\ee_i+2)\cdot k_i^{-1/p}\leq \mu_i^{-1}\cdot 2^{-i}$ for every
$i\in\{1,...,N\}$.
\end{enumerate}
We proceed to the recursive selection. As the first step is identical to the
general one, we may assume that for some $i\in\{1,...,N-1\}$ the natural numbers
$k_1,...,k_i$, the positive reals $\ee_1,...,\ee_i$ and the sets
$F_1,...,F_i$ have been selected so that conditions (C1)-(C4)
are satisfied. In particular, the number $\mu_{i+1}$ can be defined
(for the first step of the recursive selection, recall that we have
already set $\mu_1=1$). First we select $k_{i+1}\in\nn$ such that $k_{i+1}\geq N$ and
\[ 2\cdot k_{i+1}^{-1/p}\leq 2^{-1}\cdot \mu_{i+1}^{-1}\cdot 2^{-(i+1)}. \]
Next we select $\ee_{i+1}>0$ such that
\[ k_{i+1}^{1-1/p}\cdot\ee_{i+1}\leq 2^{-1}\cdot \mu_{i+1}^{-1}\cdot 2^{-(i+1)} \]
and we notice that with these choices condition (C4) is satisfied. By
(P5), the sequence $(y_n)$ is asymptotically sparse. Therefore,
it is possible to find $l\in\nn$ such that $|\{n\geq l: \|P_\sg(y_n)\|\geq\ee_{i+1}\}|\leq 1$
for every $\sg\in\nn^\nn$. We select a non-empty finite subset $F_{i+1}$
of $\nn$ such that $F_i<F_{i+1}$, $|F_{i+1}|=k_{i+1}$ and
$\min F_{i+1}\geq \max\{k_{i+1},l\}$ and we observe that with these
choices conditions (C1), (C2) and (C3) are satisfied. The recursive selection
is completed.

We define
\begin{equation} \label{6e27}
F=F_1\cup ...\cup F_N
\end{equation}
Notice that for every $n\in F$ there exists a unique $i(n)\in\{1,...,N\}$
such that $n\in F_{i(n)}$. For every $n\in F$ we define
\begin{equation} \label{6e28}
a_n= N^{-1/p} \cdot k_{i(n)}^{-1/p}.
\end{equation}
We will show that the set $F$ and the reals $(a_n)_{n\in F}$ are as desired.
\begin{claim} \label{6c26}
We have $F\in\sss_2$.
\end{claim}
\begin{proof}
Follows immediately by (C1) and (C2).
\end{proof}
\begin{claim} \label{6c27}
We have
\[ \big\| \sum_{n\in F} a_n x_n \big\|_{\ell_p}=1. \]
\end{claim}
\begin{proof}
By (P2), the sequence $(x_n)$ is $1$-equivalent to the standard unit vector basis
of $\ell_p$. Therefore,
\begin{eqnarray*}
\big\| \sum_{n\in F} a_n x_n \big\|_{\ell_p} & = &
N^{-1/p} \big\| \sum_{i=1}^N \sum_{n\in F_i} k_i^{-1/p} x_n \big\|_{\ell_p} \\
& = &  N^{-1/p} \Big( \sum_{i=1}^N \sum_{n\in F_i} k_i^{-1} \Big)^{1/p}=
N^{-1/p} \Big( \sum_{i=1}^N |F_i| \cdot k_i^{-1} \Big)^{1/p} \\
& \stackrel{\mathrm{(C2)}}{=} & N^{-1/p} \cdot N^{1/p} =1.
\end{eqnarray*}
The claim is proved.
\end{proof}
The final claim is the following.
\begin{claim} \label{6c28}
We have
\[ \big\| \sum_{n\in F} a_n y_n \big\|_{Y_p}\leq\delta. \]
\end{claim}
For the proof of Claim \ref{6c28} we need to do some preparatory work.
For every $i\in\{1,...,N\}$ we introduce the vector $z_i$ in $Y_p$ defined by
\begin{equation} \label{6e29}
z_i= k_i^{-1/p} \sum_{n\in F_i} y_n.
\end{equation}
Notice that
\begin{equation} \label{6e30}
\sum_{n\in F} a_n y_n = N^{-1/p} \sum_{i=1}^N z_i
\end{equation}
and that
\begin{equation} \label{6e31}
|\supp(z_i)|=\sum_{n\in F_i} |\supp(y_n)|
\end{equation}
for every $i\in\{1,...,N\}$.
\begin{subclaim} \label{6sc29}
Let $i\in\{1,...,N\}$. Then the following are satisfied.
\begin{enumerate}
\item[(i)] We have $\|z_i\|\leq\Theta$.
\item[(ii)] For every segment $\seg$ of $\bt$ we have
$\|P_{\seg}(z_i)\|\leq \mu_i^{-1}\cdot 2^{-i}$.
\end{enumerate}
\end{subclaim}
\begin{proof}
(i) By (P4), the sequences $(y_n)$ and $\big(T(x_n)\big)$ are $\Theta$-equivalent.
Hence,
\begin{eqnarray*}
\|z_i\| & \leq & \Theta \big\| T\Big( \sum_{n\in F_i} k_i^{-1/p} x_n\Big)\big\|_{Y_p}
\stackrel{\mathrm{(P1)}}{\leq} \Theta \big\| \sum_{n\in F_i} k_i^{-1/p} x_n\big\|_{\ell_p} \\
& \stackrel{\mathrm{(P2)}}{=} & \Theta \Big( \sum_{n\in F_i} k_i^{-1}\Big)^{1/p} =
\Theta \big( |F_i| \cdot k_i^{-1}\big)^{1/p} \stackrel{\mathrm{(C2)}}{=} \Theta.
\end{eqnarray*}
\noindent (ii) Fix a segment $\seg$ of $\bt$. Clearly we may assume that $\seg$ is
non-empty. We pick $\sg\in\nn^\nn$ such that $\seg\subseteq \sg$ and we notice that
$\|P_\seg(y)\|\leq \|P_{\sg}(y)\|$ for every vector $y$ in $Y_p$. Therefore, it is enough
to show that $\|P_{\sg}(z_i)\|\leq\mu_i^{-1}\cdot 2^{-i}$. By (P3), we see that
$\|P_\sg(y_n)\|\leq 2$ for every $n\in\nn$. Hence,
\[ \big\| P_{\sg}\Big( k_i^{-1/p} \sum_{n\in F_i} y_n\Big)\big\|
\stackrel{\mathrm{(C3)}}{\leq} \frac{(|F_i|-1)\ee_i+2}{k_i^{1/p}}
\stackrel{\mathrm{(C2)}}{=} \frac{(k_i-1)\ee_i+2}{k_i^{1/p}}
\stackrel{\mathrm{(C4)}}{\leq} \mu_i^{-1}\cdot 2^{-i} \]
and the result follows.
\end{proof}
We are ready to proceed to the proof of Claim \ref{6c28}.
\begin{proof}[Proof of Claim \ref{6c28}]
We set
\begin{equation} \label{6e32}
z=z_1+...+z_N.
\end{equation}
By the choice of $N$ in (\ref{6e25}) and equality (\ref{6e30}), it is enough
to show that
\begin{equation} \label{6e33}
\|z\|\leq N^{1/q} \cdot (2\Theta).
\end{equation}
By (\ref{5e18}), we see that for every vector $y$ in $Y_p$ we have
\[ \|y\|=\sup\Big\{ \Big( \sum_{j=1}^d \|P_{\seg_j}(y)\|^q \Big)^{1/q} \Big\} \]
where the above supremum is taken over all families of pairwise incomparable non-empty
segments of $\bt$. Therefore, it is enough to consider an arbitrary family $(\seg_j)_{j=1}^d$
of pairwise incomparable non-empty segments of $\bt$ and show that
\begin{equation} \label{6e34}
\sum_{j=1}^d \|P_{\seg_j}(z)\|^q \leq N \cdot (2\Theta)^q.
\end{equation}
So, fix such a family $(\seg_j)_{j=1}^d$. We may assume that for every $j\in\{1,...,d\}$
there exists $i\in\{1,...,N\}$ such that $\seg_j\cap \supp(z_i)\neq\varnothing$.
We define recursively a partition $(\Delta_i)_{i=1}^N$ of $\{1,...,d\}$ by the rule
$\Delta_1=\big\{ j\in\{1,...,d\}: \seg_j\cap \supp(z_1)\neq\varnothing\big\}$ and
\[ \Delta_i = \Big\{ j\in \{1,...,d\}\setminus \Big( \bigcup_{m=1}^{i-1} \Delta_m\Big):
\seg_j\cap \supp(z_i)\neq\varnothing\Big\}.\]
for every $i\in\{2,...,N\}$. The segments $(\seg_j)_{j=1}^d$ are pairwise incomparable
and a fortiori disjoint. Therefore, by equality (\ref{6e31}), for every $i\in\{1,...,N\}$
we have
\begin{equation}
\label{6e35} |\Delta_i|\leq \sum_{n\in F_i} |\supp(y_n)|.
\end{equation}

Fix $i\in\{1,...,N\}$. Let $j\in \Delta_i$ be arbitrary. Notice that if $m\in\{1,...,N\}$
with $m<i$, then $P_{\seg_j}(z_m)=0$. Therefore,
\begin{equation} \label{6e36}
\|P_{\seg_j}(z)\| = \|P_{\seg_j}(z_i +...+ z_N)\|
\leq \|P_{\seg_j}(z_i)\|+ \sum_{l=i+1}^N \|P_{\seg_j}(z_l)\|.
\end{equation}
Let $l\in\{i+1,...,N\}$ be arbitrary. By (\ref{6e26}) and (\ref{6e35}), we have
\[|\Delta_i|^{1/q} \leq \Big( \sum_{n\in F_i} |\supp(y_n)|\Big)^{1/q} \leq \mu_l\]
while, by part (ii) of Subclaim \ref{6sc29}, we have
$\|P_{\seg_j}(z_l)\|\leq \mu_l^{-1} \cdot 2^{-l}$.
By plugging the previous two estimates into (\ref{6e36}), we get that
\begin{equation} \label{6e37}
\|P_{\seg_j}(z)\| \leq \|P_{\seg_j}(z_i)\|+ |\Delta_i|^{-1/q}\cdot 2^{-i}.
\end{equation}
Using the fact that $q\geq 2$ and that $(a+b)^q\leq 2^{q-1} a^q+ 2^{q-1}b^q$ for every
pair $a$ and $b$ of positive reals, inequality (\ref{6e37}) yields that
\[ \sum_{j\in\Delta_i} \|P_{\seg_j}(z)\|^q \leq
2^{q-1} \Big( \sum_{j\in\Delta_i} \|P_{\seg_j}(z_i)\|^q\Big) + 2^{q-1}\cdot 2^{-i}.\]
The family $(\seg_j)_{j\in\Delta_i}$ consists of pairwise incomparable non-empty
segments of $\bt$. Therefore, by part (i) of Subclaim \ref{6sc29}, we get that
\begin{equation} \label{6e38}
\sum_{j\in\Delta_i} \|P_{\seg_j}(z)\|^q \leq
2^{q-1} \|z_i\|^q + 2^{q-1}\cdot 2^{-i} \leq 2^{q-1}\cdot \Theta^q+ 2^{q-1}\cdot 2^{-i}.
\end{equation}

Summing up, we conclude that
\[ \sum_{j=1}^d \|P_{\seg_j}(z)\|^q = \sum_{i=1}^N \sum_{j\in\Delta_i} \|P_{\seg_j}(z)\|^q
\stackrel{(\ref{6e38})}{\leq} 2^{q-1}\cdot (N\cdot \Theta^q+1)\leq N \cdot (2\Theta)^q\]
and the result follows.
\end{proof}
As we have already indicated, having completed the proof of Claim \ref{6c28},
the proof of part (ii) of Theorem \ref{1t4} is completed.

Finally, we notice that the map $\varrho$ is not a co-analytic rank on
$\sss\sss(\ell_p,Y_p)$. For if not, by part (ii) of Theorem
\ref{1t4} and \cite[Theorem 35.23]{Kechris}, we would get that
$\sss\sss(\ell_p,Y_p)$ is a Borel subset of $\llll(\ell_p,Y_p)$.
This contradicts part (i) of Theorem \ref{1t4}.

The proof of Theorem \ref{1t4} is completed.


\end{document}